\documentclass[11pt]{amsart}
\usepackage{txfonts}
\usepackage{mathrsfs}
\usepackage{amsmath}
\usepackage[pagewise]{lineno}
\allowdisplaybreaks[4]
\usepackage{}
\usepackage{xypic}
\usepackage{amsfonts}
\usepackage{amssymb}
\usepackage{bbm}

\usepackage[colorlinks=true, citecolor=blue]{hyperref}


\newtheorem{thm}{Theorem}[section]
\newtheorem{cor}[thm]{Corollary}
\newtheorem{lem}[thm]{Lemma}

\newtheorem{prop}[thm]{Proposition}
\newtheorem{rem}[thm]{Remark}
\numberwithin{equation}{section}

\newcommand{\be}{\begin{equation}}
\newcommand{\ee}{\end{equation}}
\newcommand{\bes}{\begin{eqnarray}}
\newcommand{\ees}{\end{eqnarray}}
\newcommand{\bess}{\begin{eqnarray*}}
\newcommand{\eess}{\end{eqnarray*}}

\newcommand{\bali}{\begin{align}}
\newcommand{\eali}{\end{align}}

\begin{document}
\title[Gauge invariants raised from commutators of unimodular Hopf algebras]{Gauge invariants raised from commutators of unimodular Hopf algebras}
%\thanks{This work was supported by }
\author{Zhihua Wang}
\address{Z. Wang\newline School of Statistics and Data Science, Taizhou University,
Taizhou 225300, China}
\email{mailzhihua@126.com}
%\author{Gongxiang Liu}
%\address{G. Liu\newline Department of Mathematics, Nanjing University,
%Nanjing 210093, China}
%\email{gxliu@nju.edu.cn}
\author{Libin Li}
\address{L. Li\newline School of Mathematical Science, Yangzhou University, Yangzhou 225002, China}
\email{lbli@yzu.edu.cn}
\date{}
\subjclass[2010]{16T05, 18D10}
\keywords{Hopf algebra, commutator, gauge invariant}

\begin{abstract}
This paper studies commutators in finite-dimensional unimodular Hopf algebras and invariants derived from them. For a unimodular Hopf algebra 
$H$ with the distinguished group-like element $g$, the commutator of $a,b\in H$ is defined by $\{a,b\}=a_{1}S^{2}(b_{1})S(a_{2})S^{-1}(b_{2})g^{-1}$. 
We show that the span of all commutators is a left coideal and generates a normal left coideal subalgebra $N$ of $H$ for which $H/HN^{+}$ is commutative.
The main result asserts that for any central element $a$ and any integral $\Lambda$ of $H$, the commutator $\{a,\Lambda\}$ and the value $\lambda(\{a,\Lambda\})$ for a right integral $\lambda$ of $H^*$ are invariant under twisting. This yields several gauge invariants, including ratios of characters of iterated commutators with $\Lambda$. In the semisimple characteristic-zero case, explicit descriptions of these invariants are given. As applications, we compute some invariants for the group algebra of a finite group, its Drinfeld double, and the restricted quantum group $\overline{U}_{q}(sl_{2})$.
\end{abstract}
\maketitle
\section{\bf Introduction}
Commutators originated over 100 years ago as a by-product of computing group characters of nonabelian groups. They are now an established and immensely useful
tool in all of group theory. Commutators became objects of interest in their own right soon after their introduction. 
The notation of a commutator for any Hopf algebra $H$ has been introduced by Cohen and Westreich in \cite{Cohen}. They defined Hopf-algebraic commutators by \begin{equation}\label{s}\{a,b\}= a_1b_1S(a_2)S(b_2) \end{equation} for all $a,b\in H$. If $H$ is semisimple, they showed in  \cite[Proposition 2.2]{Cohen} that the subalgebra $H'$ of $H$ generated by all commutators of $H$ coincides with the commutator subalgebra of $H$ defined in \cite{Bur} as the smallest normal left coideal subalgebra with the property that the corresponding quotient is a commutative Hopf algebra. They introduced central elements $z_n$ that recover classical counting functions for groups \cite{Cohen}; they further studied nilpotency via commutator matrices and probabilistic methods \cite{CW}. 

There is another way to define commutators for Hopf algebras. In this paper, for an arbitrary finite-dimensional unimodular Hopf algebra $H$, we define the commutator of $a,b\in H$ as
\begin{equation}\label{ss}
\{a,b\}:=a_1S^2(b_1)S(a_2)S^{-1}(b_2)g^{-1},
\end{equation}
where $g$ is the distinguished group-like element of $H$. If $H$ is semisimple, this is exactly the form of (\ref{s}). 

For the commutator given in (\ref{ss}), we study its invariant under twisting, where the twisted Hopf algebra $H^J$ shares the same algebra as $H$ but has deformed comultiplication and antipode (see e.g. \cite{AEGN, Drin}). For any central element $a\in H$ and any integral $\Lambda\in H$, we show that the commutator $\{a,\Lambda\}$ and the value $\lambda(\{a,\Lambda\})$ for a right integral $\lambda\in H^*$ are both invariant under twisting (Theorem \ref{t1}). In particular, the sequence $\{\Lambda^{\{m\}}\mid m>0\}$ is invariant under twisting, where  $\Lambda^{\{m\}}$ denotes the $m$-fold iterated commutator with respect to $\Lambda$. 

After that we use these twisting invariants to obtain several categorical gauge invariants of $H$ (Theorem \ref{prop112}). Under the assumption that $\chi_W(\Lambda^{\{m\}})\neq0$, we show that the ratios of values of characters on iterated commutators
\[
\frac{\chi_V(\Lambda^{\{m\}})}{\chi_W(\Lambda^{\{m\}})}
\quad \text{and}\quad
\frac{(\chi_{V}(\Lambda^{\{n\}}))^m}{(\chi_{W}(\Lambda^{\{m\}}))^n}
\]
are gauge invariants of $H$, for any finite-dimensional $H$-module $V$ and any positive integer $n$. Let $\Lambda^{[n]}$ be the $n$-th Sweedler power of $\Lambda$ for any integer $n$.
We also obtain that the ratio
\[
\frac{(\chi_V(\Lambda^{[n]}))^m}{\chi_W(\Lambda^{\{m\}})}
\]
is a  gauge invariant of $H$,  for any finite-dimensional $H$-module $V$.
 %This invariant generalizes earlier $n$-th indicator from the semisimple case to the unimodular setting.

If $H$ is a semisimple Hopf algebra over an algebraically closed field $\mathbbm{k}$ of characteristic 0, then $S^2=\operatorname{id}$ and the distinguished group-like element $g=1$. In this case, the commutator is reduced to be $\{a,b\}=a_1b_1S(a_2)S(b_2)$ and it has been studied in \cite{CW}, where $\Lambda^{\{m\}}$ is denoted to be $\gamma_{m-1}$. The expression of $\Lambda^{\{m\}}$ in terms of central primitive idempotents is given in \cite[Proposition 2.2]{CW}, where the coefficients are given by entries of so called commutator matrix. The commutators  $\Lambda^{\{m\}}$ can be used to determine the nilpotency of $H$. Under a minor assumption on the Grothendieck ring of $H$, the authors proved in \cite[Theorem 2.6]{CW} that $H$ is nilpotent if and only if $\gamma_m=1$ for some $m$. We refer to \cite[Section 2]{CW} for more details about commutators $\Lambda^{\{m\}}$. 

In the semisimple characteristic-zero case, we choose the integral $\Lambda\in H$ to be idempotent. In this case,  $\chi_V(\Lambda^{\{m\}})$ is a gauge invariant of $H$ for any  finite-dimensional $H$-module $V$ and any $m>0$. Moreover, we can express $\chi_V(\Lambda^{\{m\}})$ via the fusion rules of the representation category of $H$-modules (as a fusion category). Especially, we find out that  $$\chi_{\operatorname{reg}}(\Lambda^{\{m\}})=\chi_{\operatorname{ad}}(\Lambda^{\{m-1\}})\ \text{for}\ m>1,$$ where $\chi_{\operatorname{reg}}$ and $\chi_{\operatorname{ad}}$ are the characters of the (left) regular representation and the (left) adjoint representation of $H$ respectively.

As applications, we compute invariants explicitly for the group algebra of a finite group, the Drinfeld double of a finite group, and the restricted quantum group $\overline{U}_q(sl_2)$ at a primitive $2p$-th root of unity. For the latter, we obtain the values for $\chi_V(\Lambda^{\{2\}})$ and $\chi_{\operatorname{reg}}(\Lambda^{\{2\}})$ explicitly, where $V$ is a simple module of dimension $p$ generated by a highest weight vector. We find out that their ratio is $2p$,  a trivial invariant of  $\overline{U}_q(sl_2)$.

The paper is organized as follows. Section~2 recalls preliminaries on unimodular Hopf algebras, integrals, Radford's trace formulae and twists. Section~3 defines commutators and establishes their basic properties. Section~4 provides several twisting invariants and constructs the associated gauge invariants. Section~5 treats the semisimple case. Sections~6, 7, and~8 are devoted to examples: finite groups, Drinfeld doubles, and restricted quantum groups, respectively.

\section{\bf Preliminaries}
Throughout this paper, $H$ is a finite-dimensional unimodular Hopf algebra over a field $\mathbbm{k}$, with counit $\varepsilon$, comultiplication $\Delta$ and antipode $S$. We denote by $S^{-1}$ the inverse of $S$ under composition. We use the Sweedler notation $\Delta(a)=a_{1}\otimes a_{2}$ for $a\in H$. We denote by $\Lambda$ and $\lambda$ the left and right integrals of $H$ and $H^*$, respectively, normalized by $\lambda(\Lambda)=1$. We denote $g$ the distinguished group-like element of $H$ and $\alpha$ the distinguished group-like element of $H^*$, respectively. 
Since $H$ is unimodular, the left integral $\Lambda$ is also a right integral, equivalently, the distinguished group-like element $\alpha$ equals $\varepsilon$. In this case, the Radford's $S^4$-formula is $S^4(a)=gag^{-1}$ for $a\in H$. The center of $H$ is denoted by $Z(H)$. The category of finite-dimensional left $H$-modules is denoted by $H$-mod, which is a rigid tensor category. The characters of the (left) regular representation and the (left) adjoint representation of $H$ are denoted by $\chi_{\operatorname{reg}}$ and $\chi_{\operatorname{ad}}$ respectively.
We refer to \cite{Mon} for basic theory of Hopf algebras.

The following equalities are of fundamental importance for a  finite-dimensional unimodular Hopf algebra $H$. 
The first equality can be found in \cite[Lemma 1.2]{LR}:
\begin{equation}\label{equ004}S(a)\Lambda_{1}\otimes \Lambda_{2}=\Lambda_{1}\otimes a\Lambda_{2},\ \text{for}\ a\in H.\end{equation}
The second one can be found in \cite{KMN}:
\begin{equation}\label{equ005}\Lambda_{1}a\otimes \Lambda_{2}=\Lambda_{1}\otimes \Lambda_{2}S(a),\ \text{for}\ a\in H.\end{equation}
The third one can be found in \cite[Theorem 3(d)]{Rad}: 
\begin{equation}\label{0e}\Lambda_{2}\otimes \Lambda_{1}=\Lambda_{1}\otimes S^2(\Lambda_{2})g.\end{equation}
The integral $\lambda$ satisfies the following properties:
\begin{gather}
\label{equ4004}\lambda(ab)=\lambda(S^2(b)a)\ \text{for}\ a,b\in H,\\
\label{equ4004444}a=\lambda(a\Lambda_{1})S(\Lambda_{2})=\lambda(S(\Lambda_2)a)\Lambda_{1}\ \text{for}\ a\in H,
\end{gather}
where the former follows from \cite[Theorem 3(a)]{Rad} and the latter from \cite[Eq. (1)]{Rad}.
For any  $\mathbbm{k}$-linear map $f:H\rightarrow H$, the trace of $f$ can be described by Radford's trace formula (see \cite[Theorem 2]{Rad}): \begin{equation}\label{equ400444}\text{Tr}(f)=\lambda(S(\Lambda_{2})f(\Lambda_{1})).\end{equation}

Recall that a normalized twist for a finite-dimensional Hopf algebra $H$ is an invertible element $J\in H\otimes H$ that satisfies $(\varepsilon\otimes \operatorname{id})(J)=(\operatorname{id}\otimes\varepsilon)(J)=1$ and
\begin{equation}\label{qq1}(\Delta\otimes \operatorname{id})(J)(J\otimes1)=(\operatorname{id}\otimes\Delta)(J)(1\otimes J).\end{equation}
We write $J=J^{(1)}\otimes J^{(2)}$ and $J^{-1}=J^{-(1)}\otimes J^{-(2)}$, where the summation is understood. We also write $J_{21}=J^{(2)}\otimes J^{(1)}$.

Given a normalized twist $J$ for $H$, one can define a new Hopf algebra $H^{J}$. It has the same algebra structure and counit as $H$, while its comultiplication $\Delta^J$ and antipode $S^J$ are given respectively by
$$\Delta^J(a)=J^{-1}\Delta(a)J,\ \ \ \
S^J(a)=Q^{-1}_JS(a)Q_J\ \text{for}\ a\in H,$$
where $Q_J=S(J^{(1)})J^{(2)}$, which is invertible with the inverse  $Q_J^{-1}=J^{-(1)}S(J^{-(2)})$.

The element $Q_J$ satisfies the following identity (see \cite[Eq. (5)]{AEGN}):
\begin{equation}\label{equ002}
\Delta(Q_J)=(S\otimes S)(J^{-1}_{21})(Q_J\otimes Q_J)J^{-1}.
\end{equation}
This implies that
\begin{equation}\label{equu00003}
\Delta(Q_J^{-1})=J(Q_J^{-1}\otimes Q_J^{-1})(S\otimes S)(J_{21}).
\end{equation}
The twist $J$ of $H$ satisfies the following properties (see \cite[Lemma 2.4]{AEGN}):
\begin{gather}
\label{eqq11}J^{-(1)}\otimes S(J^{-(2)})Q_J=J^{(1)}_{(1)}\otimes S(J^{(1)}_{(2)})J^{(2)},\\
\label{eqq1}J^{(1)}\otimes S^{-1}(Q_J^{-1})S^{-1}(J^{(2)})=J^{-(1)}_{(1)}\otimes J^{-(2)}S^{-1}(J^{-(1)}_{(2)}).
\end{gather}

\section{\bf Commutators for Hopf algebras}

For a unimodular Hopf algebra $H$ over a field $\mathbbm{k}$,
we define the commutator of  $a,b\in H$ by
$$\{a,b\}:=a_{1}S^2(b_{1})S(a_{2})S^{-1}(b_{2})g^{-1}.$$
In particular, $\{1,1\}=g^{-1}$.
If $H$ is semisimple and $\operatorname{char}(\mathbbm{k})=0$, then $S^2=\mathrm{id}$ and $g=1$. In this case, the commutator is $\{a,b\}=a_{1}b_{1}S(a_{2})S(b_{2}),$ which has been investigated in \cite{Cohen, CW, CW1}. 

\begin{lem}\label{lem1}
Let $H$ be a unimodular Hopf algebra over a field $\mathbbm{k}$.
\begin{enumerate}
  \item For any integer $m$, $g^m$ is a power of the commutator $\{1,1\}$.
  \item For $a,b\in H$, $ab=\{a_1,S^{-2}(b_1)\}b_2S^4(a_2)g$. 
  \item For $h\in H$, $S(h_2)S^4(h_1)=\{S(h_2),S^2(h_1)\}g$.
\end{enumerate}
\end{lem}
\proof
(1) For any integer $m$, if $m=qr-s$, where $r$ is the order of $g$ and $0\leq s\leq r-1$, then $g^m=g^{-s}=\{1,1\}^{s}$.

(2)  For $a,b\in H$, we have
\begin{align*}
\{a_1,S^{-2}(b_1)\}b_2S^4(a_2)g&=a_1b_1S(a_2)S^{-3}(b_2)g^{-1}b_3S^{4}(a_3)g\\
&=a_1b_1S(a_2)S^{-3}(b_2)S^{-4}(b_3S^{4}(a_3))\\
&=a_1b_1S(a_2)S^{-3}(b_2)S^{-4}(b_3)a_3\\
&=a_1bS(a_2)a_3\\
&=ab.
\end{align*}

(3) It is a consequence of Part (2). Indeed, 
\begin{align*}
S(h_2)S^4(h_1)&=\{S(h_2)_1,S^{-2}(S^4(h_1)_1)\}S^4(h_1)_2S^4(S(h_2)_2)g\\
&=\{S(h_4),S^{2}(h_1)\}S^4(h_2)S^5(h_3)g\\
&=\{S(h_2),S^{2}(h_1)\}g.
\end{align*}
This completes the proof.
\qed

Recall that a left coideal of $H$ is a $\mathbbm{k}$-linear subspace $I\subseteq H$ such that $\Delta(I)\subseteq H\otimes I$.
Set $$\text{Com}=\text{span}_{\mathbbm{k}}\{\{a,b\}\mid a,b\in H\}.$$ 
\begin{lem}
Let $H$ be a unimodular Hopf algebra over a field $\mathbbm{k}$. Then $\operatorname{Com}$ is a left coideal of $H$. 
\end{lem}
\proof For $a,b\in H$, 
\begin{align*}
\Delta(\{a,b\})&=\Delta(a_{1}S^2(b_{1})S(a_{2})S^{-1}(b_{2})g^{-1})\\
&=a_{1}S^2(b_{1})S(a_{4})S^{-1}(b_{4})g^{-1}\otimes a_{2}S^2(b_{2})S(a_{3})S^{-1}(b_{3})g^{-1}\\
&=a_{1}S^2(b_{1})S(a_{3})S^{-1}(b_{3})g^{-1}\otimes \{a_2,b_2\}\in H\otimes \text{Com}.
\end{align*}
This completes the proof.
\qed

Let $N$ denote the subalgebra of $H$ generated by $\operatorname{Com}$. Since $\operatorname{Com}$ is a left coideal of $H$, it follows that $N$ is a left coideal subalgebra of $H$.
Note that $g^{-1}=\{1,1\}$. If the order of $g$ is $r$, then $1=\{1,1\}^r\in N$.

\begin{prop}Let $H$ be a unimodular Hopf algebra over a field $\mathbbm{k}$.
The left coideal subalgebra $N$ is normal, i.e., $N$ is stable under the left adjoint action of $H$.
\end{prop}
\proof
Since $H$ is an $H$-module algebra under the left adjoint action, it is enough to check that $h_1\{x,y\}S(h_2)\in N$ for $h,x,y\in H$. Indeed, taking $h_1\{x,y\}$ as $a$ and $S(h_2)$ as $b$, and applying Lemma \ref{lem1}(2), we obtain
\begin{align*}
h_1\{x,y\}S(h_2)&=\{h_1\{x,y\}_1,S^{-1}(h_4)\}S(h_3)S^4(h_2)S^4(\{x,y\}_2)g\\
&=\{h_1\{x,y\}_1,S^{-1}(h_4)\}\{S(h_3),S^2(h_2)\}gS^4(\{x,y\}_2)g\ \text{by\ Lemma}\ \ref{lem1} (3)\\
&=\{h_1\{x,y\}_1,S^{-1}(h_4)\}\{S(h_3),S^2(h_2)\}g^2\{x,y\}_2\in N,
\end{align*}
where the last term belongs to $N$ since $g^2\in N$ and $\{x,y\}_2\in \operatorname{Com}$.
\qed

\begin{prop}Let $H$ be a unimodular Hopf algebra over a field $\mathbbm{k}$. The Hopf quotient $H/HN^{+}$ is commutative, where $N^+=N\bigcap \ker \varepsilon$. 
\end{prop}
\proof For $a,b\in H$, it follows from Lemma \ref{lem1}(2) that 
\begin{align*}
\overline{ab}&=\overline{\{a_1,S^{-2}(b_1)\}b_2S^4(a_2)g}
=\overline{\{a_1,S^{-2}(b_1)\}} \cdot \overline{b_2S^4(a_2)g}\\
&=\varepsilon(\{a_1,S^{-2}(b_1)\})\overline{b_2S^4(a_2)g}
=\overline{bS^4(a)g}=\overline{bga}
=\overline{ba}\ \ \ \text{by\ Lemma \ref{lem1}(1)} 
\end{align*}
This completes the proof.
\qed

%If $\sigma\in G(H^{*})$, then  one computes that $\sigma\rightharpoonup\{a,b\}=\sigma(g^{-1})\{a,b\}$ for all $a,b\in H$. Indeed,
%\begin{align*}
%\sigma\rightharpoonup\{a,b\}&=\sigma(\{a,b\}_2)\{a,b\}_1\\
%&=\sigma(a_2S^2(b_2)S(a_3)S^{-1}(b_3)g^{-1})a_1S^2({b_1})S(a_4)S^{-1}(b_4)g^{-1}\\
%&=\sigma(S^2(b_2)S^{-1}(b_3)g^{-1})a_1S^2({b_1})S(a_2)S^{-1}(b_4)g^{-1}\\
%&=\sigma(S^2(b_2)g^{-1}S^{3}(b_3))a_1S^2({b_1})S(a_2)S^{-1}(b_4)g^{-1}\\
%&=\sigma(g^{-1})a_1S^2({b_1})S(a_2)S^{-1}(b_2)g^{-1}\\
%&=\sigma(g^{-1})\{a,b\}.
%\end{align*}

If $a\in Z(H)$, it is similar to \cite[Lemma 2.3]{Cohen} that $\{a,\Lambda\}\in Z(H)$. This is proved in the  following proposition:
\begin{prop}\label{p1}Let $H$ be a unimodular Hopf algebra over a field $\mathbbm{k}$ and $\Lambda$ an integral of $H$. 
\begin{enumerate}
  \item For any $a\in H$, we have $\{a,\Lambda\}=a_1\Lambda_2S(a_2)S^{-1}(\Lambda_1)$.
  \item If $a\in Z(H)$, then $\{a,\Lambda\}\in Z(H)$.
\end{enumerate}
\end{prop}
\proof (1) Note that $S^4(a)=gag^{-1}$ for any $a\in H$. We have
\begin{align*}\{a,\Lambda\}&=a_{1}S^2(\Lambda_{1})S(a_{2})S^{-1}(\Lambda_{2})g^{-1}
=a_{1}S^4(\Lambda_{2})gS(a_{2})S^{-1}(\Lambda_{1})g^{-1}\ \text{by}\ (\ref{0e})\\
&=a_{1}g\Lambda_{2}S(a_{2})S^{-1}(\Lambda_{1})g^{-1}
=a_{1}(g\Lambda_{2})S(a_{2})S^{-1}(g\Lambda_{1})\\
&=a_{1}\Lambda_{2}S(a_{2})S^{-1}(\Lambda_{1}).
\end{align*}

(2) If $a\in Z(H)$, then for all $h\in H$ we have $$h_{1}a_{1}\otimes h_{2}a_{2}=a_{1}h_{1}\otimes a_{2}h_{2}.$$
 Using Part (1), we have
\begin{align*}
h\{a,\Lambda\}
&=ha_{1}\Lambda_{2}S(a_{2})S^{-1}(\Lambda_{1})
=h_1a_{1}\Lambda_{2}S(S^{-1}(h_3)h_2a_{2})S^{-1}(\Lambda_{1})\\
&=a_1h_1\Lambda_{2}S(S^{-1}(h_3)a_2h_2)S^{-1}(\Lambda_{1})
=a_1h_1\Lambda_{2}S(h_2)S(a_2)S^{-1}(\Lambda_{1}S(h_3))\\
&=a_1h_1\Lambda_{2}S(a_2)S^{-1}(\Lambda_{1}h_2S(h_3))\ \text{by}\ (\ref{equ005})\\
&=a_1h\Lambda_{2}S(a_2)S^{-1}(\Lambda_{1})
=a_1\Lambda_{2}S(a_2)S^{-1}(\Lambda_{1})h\ \text{by}\ (\ref{equ004})\\
&=\{a,\Lambda\}h.
\end{align*}
This completes the proof.
\qed

\section{\bf Invariants raised from commutators}
In this section, we first show that the commutators $\{a,\Lambda\}$ for $a\in Z(H)$ are invariant under twisting. We then use them to construct several gauge invariants of Hopf algebra $H$.
Given a normalized twist $J$ for the unimodular Hopf algebra $H$, one obtains the twisted Hopf algebra $H^J$. Obviously, $H^J$ and $H$ share the same algebra structure (hence $Z(H^J)=Z(H)$) and the same integral $\Lambda$. If $\lambda$ is a  right integral of $H^*$ satisfying $\lambda(\Lambda)=1$, then $\lambda^J:=\lambda\leftharpoonup S(Q_J^{-1})Q_J$ is a right integral of $(H^J)^*$ satisfying $\lambda^J(\Lambda)=1$ (see \cite[Theorem 3.4]{AEGN}).
The commutator of $a$ and $b$ in $H^J$ is denoted by $\{a,b\}^J$. 

\begin{thm}\label{t1}
Let $H$ be a unimodular Hopf algebra over a field $\mathbbm{k}$ with a left integral $\Lambda$ of $H$ and a right integral $\lambda$ of $H^*$ such that $\lambda(\Lambda)=1$. Let $J$ be a normalized twist of $H$. For any $a\in Z(H)$, the commutator $\{a,\Lambda\}$ and the value $\lambda(\{a,\Lambda\})$ are both invariant under twisting, i.e.,
\begin{enumerate}
  \item We have $\{a,\Lambda\}^J=\{a,\Lambda\}$ for $a\in Z(H)$.  
  \item We have $\lambda^J(\{a,\Lambda\}^J)=\lambda(\{a,\Lambda\})$ for $a\in Z(H)$.  
\end{enumerate}
\end{thm}
\proof (1) Note that $\Delta^J(a)=J^{-1}\Delta(a)J=J^{-(1)}a_1J^{(1)}\otimes J^{-(2)}a_2J^{(2)}$. Using (\ref{equ004}) and (\ref{equ005}), we obtain  $\Delta^J(\Lambda)=J^{-(1)}\Lambda_1J^{(1)}\otimes J^{-(2)}\Lambda_2J^{(2)}=Q_J^{-1}\Lambda_1\otimes\Lambda_2Q_J$. Since $S^J(h)=Q_J^{-1}S(h)Q_J$, it follows that $(S^J)^{-1}(h)=S^{-1}(Q_JhQ_J^{-1})$ for $h\in H$. Now by Proposition \ref{p1} (1) we have 
\begin{align*}
\{a,\Lambda\}^J&=J^{-(1)}a_1J^{(1)}\Lambda_2Q_JS^J(J^{-(2)}a_2J^{(2)})(S^J)^{-1}(Q_J^{-1}\Lambda_1)\\
&=J^{-(1)}a_1J^{(1)}\Lambda_2Q_JQ_J^{-1}S(J^{-(2)}a_2J^{(2)})Q_JS^{-1}(Q_J(Q_J^{-1}\Lambda_1)Q_J^{-1})\\
&=J^{-(1)}a_1J^{(1)}\Lambda_2S(J^{(2)})S(a_2)S(J^{-(2)})Q_JS^{-1}(Q_J^{-1})S^{-1}(\Lambda_1)\\
&=J^{-(1)}a_1J^{(1)}\Lambda_2S(a_2)S(J^{-(2)})Q_JS^{-1}(Q_J^{-1})S^{-1}(J^{(2)})S^{-1}(\Lambda_1)\ \text{by}\ (\ref{equ005})\\
&=J^{(1)}_{(1)}a_1J^{-(1)}_{(1)}\Lambda_2S(a_2)S(J^{(1)}_{(2)})J^{(2)}J^{-(2)}S^{-1}(J^{-(1)}_{(2)})S^{-1}(\Lambda_1)\ \text{by}\ (\ref{eqq11})\ \text{and}\  (\ref{eqq1})\\
%&=J^{(1)}_{(1)}a_1J^{-(1)}_{(1)}\Lambda_2S(a_2)S(J^{(1)}_{(2)})J^{(2)}J^{-(2)}S^{-1}(J^{-(1)}_{(2)})S^{-1}(\Lambda_1)\ \text{by}\ (\ref{eqq1})\\
&=a_1J^{(1)}_{(1)}J^{-(1)}_{(1)}\Lambda_2S(J^{(1)}_{(2)})S(a_2)J^{(2)}J^{-(2)}S^{-1}(J^{-(1)}_{(2)})S^{-1}(\Lambda_1)\ (\text{since}\ a\in Z(H))\\
&=a_1J^{(1)}_{(1)}J^{-(1)}_{(1)}\Lambda_2S(a_2)J^{(2)}J^{-(2)}S^{-1}(J^{-(1)}_{(2)})S^{-1}(J^{(1)}_{(2)})S^{-1}(\Lambda_1)\ \text{by}\ (\ref{equ005})\\
&=a_1\Lambda_2S(a_2)S^{-1}(\Lambda_1)\\
&=\{a,\Lambda\}.
\end{align*}

(2) By Part (1), we have
\begin{align*}
\lambda^J(\{a,\Lambda\}^J)&=\lambda(S(Q_J^{-1})Q_J\{a,\Lambda\})\\
&=\lambda(S^2(J^{-(2)})S(J^{-(1)})S(J^{(1)})J^{(2)}\{a,\Lambda\})\\
&=\lambda(S^2(J^{-(2)})S(J^{(1)}J^{-(1)})\{a,\Lambda\}J^{(2)})\ (\text{since}\ \{a,\Lambda\}\in Z(H))\\
&=\lambda(S(J^{(1)}J^{-(1)})\{a,\Lambda\}J^{(2)}J^{-(2)})\ \text{by}\ (\ref{equ4004})\\
&=\lambda(\{a,\Lambda\}).
\end{align*}
This completes the proof.
\qed

In order to obtain gauge invariants derived from commutators, we begin with the following preparations. 
Let $H$ be a finite-dimensional unimodular Hopf algebra over a field $\mathbbm{k}$ with a nonzero integral $\Lambda$. Define $$ \Lambda^{\{1\}}=\Lambda\ \text{and}\ \Lambda^{\{m\}}=\{\Lambda^{\{m-1\}},\Lambda\}\ \text{for}\ m>1.$$
Then $\Lambda^{\{m\}}\in Z(H)$ for $m>0$. Similarly, in $H^J$,  we define  $\Lambda^{\{1\}^J}=\Lambda$ and $\Lambda^{\{m\}^J}=\{\Lambda^{\{m-1\}^J},\Lambda\}^J$ for $m>1$, where $\{a,b\}^J$ is the commutator of $a,b$ in $H^J$. It follows from Theorem \ref{t1}(1) that $\Lambda^{\{m\}^J}=\Lambda^{\{m\}}$ for $m>0$, i.e., the sequence $\{\Lambda^{\{m\}}\mid m>0\}$ is invariant under twisting.

\begin{thm}\label{prop112}
Let $H$ be a finite-dimensional unimodular Hopf algebra over a field $\mathbbm{k}$ with a nonzero integral $\Lambda$.
Suppose $\chi_{W}(\Lambda^{\{m\}})\neq0$ for some finite-dimensional $H$-module $W$ and some positive integer $m$.
\begin{enumerate}
\item The ratio $\chi_{V}(\Lambda^{\{m\}})/\chi_{W}(\Lambda^{\{m\}})$ is a gauge invariant of $H$ for any finite-dimensional $H$-module $V$.

\item The ratio $(\chi_{V}(\Lambda^{\{n\}}))^m/(\chi_{W}(\Lambda^{\{m\}}))^n$ is a gauge invariant of $H$ for any finite-dimensional $H$-module $V$ and any positive integer $n$.

\item The ratio $(\chi_{V}(\Lambda^{[n]}))^m/\chi_{W}(\Lambda^{\{m\}})$ is a gauge invariant of $H$ for any finite-dimensional $H$-module $V$ and any integer $n$, where $\Lambda^{[n]}$ is the $n$-th Sweedler power of $\Lambda$.
\end{enumerate}\end{thm}
\proof
Let $H$ and $H'$ be finite-dimensional unimodular Hopf algebras over a field $\mathbbm{k}$ with nonzero integrals $\Lambda$ and $\Lambda'$ respectively. If the functor $\mathcal{F}:H\text{-mod}\rightarrow H'\text{-mod}$ is an equivalence of tensor categories, it follows from \cite[Theorem 2.2]{NS1} that there exists
a normalized twist $J$ of $H$  such that $H'$ is isomorphic to $H^J$ as bialgebras. Let $\sigma:H'\rightarrow H^J$ be such an isomorphism. Then $\sigma$ is
automatically a Hopf algebra isomorphism. Therefore, for any $m>0$, we have
$$\sigma(\Lambda'^{\{m\}})=\sigma(\Lambda')^{\{m\}^J}=\sigma(\Lambda')^{\{m\}},$$ where the second equality follows from  Theorem \ref{t1}(1). 
%Suppose $\sigma(\Lambda')=\mu\Lambda$ for a nonzero scalar $\mu\in\mathbbm{k}$. Then
%\begin{equation}\label{qu01}\sigma(P'_n(\Lambda'))=P^J_n(\sigma(\Lambda'))=\mu P^J_n(\Lambda)=\mu TP_n(\Lambda),\end{equation}where $T=J^{-(1)}\alpha(J^{-(2)})$.
The isomorphism $\sigma$ induces a $\mathbbm{k}$-linear equivalence $(-)^\sigma:H\text{-mod}\rightarrow H'\text{-mod}$ as follows:  for any finite-dimensional $H$-module $V$, $V^{\sigma}=V$ as $\mathbbm{k}$-linear space with the $H'$-module structure given by $a'v=\sigma(a')v$ for $a'\in H'$, $v\in V$, and $f^\sigma=f$ for any morphism $f$ in $H$-mod.
Thus, $$\chi_{V^{\sigma}}(a')=\chi_V(\sigma(a'))\ \text{for}\ a'\in H'.$$
Moreover, the equivalence  $\mathcal{F}$ is naturally isomorphic to the $\mathbbm{k}$-linear equivalence $(-)^\sigma$ (see \cite[Theorem 1.1]{KMN}).
Therefore, $$\chi_{\mathcal{F}(V)}(a')=\chi_{V^{\sigma}}(a')\ \text{for}\ a'\in H'.$$
Taking $a'=\Lambda'^{\{m\}}$, we have
\begin{equation}\label{qqu}\chi_{\mathcal{F}(V)}(\Lambda'^{\{m\}})=\chi_{V^{\sigma}}(\Lambda'^{\{m\}})=\chi_{V}(\sigma(\Lambda'^{\{m\}}))=\chi_{V}(\sigma(\Lambda')^{\{m\}}).\end{equation}
We set $\sigma(\Lambda')=\mu\Lambda$ for some nonzero scalar $\mu$. Then $\sigma(\Lambda')^{\{m\}}=\mu^m \Lambda^{\{m\}}$ and (\ref{qqu}) becomes
\begin{equation}\label{qqqu}\chi_{\mathcal{F}(V)}(\Lambda'^{\{m\}})=\mu^m\chi_{V}(\Lambda^{\{m\}})\end{equation}for any finite-dimensional $H$-module $V$ and $m>0$.

(1)  If  $\chi_{W}(\Lambda^{\{m\}})\neq0$ for some finite-dimensional $H$-module $W$ and some positive integer $m$, then the ratios
$$\frac{\chi_{\mathcal{F}(V)}(\Lambda'^{\{m\}})}{\chi_{\mathcal{F}(W)}(\Lambda'^{\{m\}})}
=\frac{\mu^m\chi_{V}(\Lambda^{\{m\}})}{\mu^m\chi_{W}(\Lambda^{\{m\}})}=\frac{\chi_{V}(\Lambda^{\{m\}})}{\chi_{W}(\Lambda^{\{m\}})}$$
show that  $\chi_{V}(\Lambda^{\{m\}})/\chi_{W}(\Lambda^{\{m\}})$ is a gauge invariant of $H$ for any finite-dimensional $H$-module $V$.

(2) The ratios
$$\frac{(\chi_{\mathcal{F}(V)}(\Lambda'^{\{n\}}))^m}{(\chi_{\mathcal{F}(W)}(\Lambda'^{\{m\}}))^n}=\frac{\mu^{mn}(\chi_{V}(\Lambda^{\{n\}}))^m}{\mu^{mn}(\chi_{W}(\Lambda^{\{m\}}))^n}
=\frac{(\chi_{V}(\Lambda^{\{n\}}))^m}{(\chi_{W}(\Lambda^{\{m\}}))^n}
$$ show that $(\chi_{V}(\Lambda^{\{n\}}))^m/(\chi_{W}(\Lambda^{\{m\}}))^n$ is a gauge invariant of $H$ for any finite-dimensional $H$-module $V$ and any positive integer $n$.

(3) Using  (\ref{qqqu}) and the equality $\chi_{\mathcal{F}(V)}(\Lambda'^{[n]})=\mu\chi_{V}(\Lambda^{[n]})$ from \cite{W}, we have 
$$\frac{(\chi_{\mathcal{F}(V)}(\Lambda'^{[n]}))^m}{\chi_{\mathcal{F}(W)}(\Lambda'^{\{m\}})}
=\frac{\mu^m(\chi_{V}(\Lambda^{[n]}))^m}{\mu^m\chi_{W}(\Lambda^{\{m\}})}
=\frac{(\chi_{V}(\Lambda^{[n]}))^m}{\chi_{W}(\Lambda^{\{m\}})}.$$
Hence $(\chi_{V}(\Lambda^{[n]}))^m/\chi_{W}(\Lambda^{\{m\}})$ is a gauge invariant of $H$ for any finite-dimensional $H$-module $V$.
\qed

If, moreover,  $H$ is semisimple, then the trivial $H$-module $\mathbbm{k}$ satisfies $\chi_{\mathbbm{k}}(\Lambda^{\{m\}})=\varepsilon(\Lambda)^m\neq0$ for $m>0$. Thus, $\chi_{V}(\Lambda^{\{m\}})/\chi_{\mathbbm{k}}(\Lambda^{\{m\}})=\chi_{V}((\Lambda/\varepsilon(\Lambda))^{\{m\}})$ is a gauge invariant of $H$, where $\Lambda/\varepsilon(\Lambda)$ is an idempotent integral of $H$. We summarize it as follows:  

%As a result, we have the following corollary:
\begin{cor}
Let $H$ be a finite-dimensional semisimple Hopf algebra over a field $\mathbbm{k}$ with an idempotent integral $\Lambda$. For any finite-dimensional $H$-module $V$ and any $m>0$, $\chi_{V}(\Lambda^{\{m\}})$ is a gauge invariant of $H$. In particular, if $\lambda$ is a right integral of $H^*$ normalized by $\lambda(\Lambda)=1$, then $\lambda=\chi_{\text{reg}}$. In this case, $\lambda(\Lambda^{\{m\}})$ is a gauge invariant of $H$ for $m>0$.
\end{cor}

\begin{prop}
Let $H$ be a finite-dimensional unimodular Hopf algebra over a field $\mathbbm{k}$ with a nonzero integral $\Lambda$. Let $\lambda$ be a right integral of $H^*$ such that $\lambda(\Lambda)=1$.
\begin{enumerate}
\item  $S(\Lambda^{\{2\}})=\Lambda^{\{2\}}$. Hence $\chi_{V^*}(\Lambda^{\{2\}})=\chi_V(\Lambda^{\{2\}})$ for any  finite-dimensional $H$-module $V$.
\item  $\lambda(\Lambda^{\{2\}})=\chi_{\operatorname{ad}}(\Lambda)$.
\end{enumerate}
\end{prop}
\proof
Note that $\lambda(\Lambda)=1=\lambda(S(\Lambda))$ (see \cite[Proposition 1(e)]{Rad}). It follows that $S(\Lambda)=\Lambda$. Set $\widetilde{\Lambda}=\Lambda$.

(1) By Proposition \ref{p1}(1), we have
\begin{align*}
S(\Lambda^{\{2\}})&=S(\Lambda_1\widetilde{\Lambda}_2S(\Lambda_2)S^{-1}(\widetilde{\Lambda}_1))=\widetilde{\Lambda}_1S^2(\Lambda_2)S(\widetilde{\Lambda}_2)S(\Lambda_1)\\
&=\widetilde{\Lambda}_1S^2(\Lambda_2)S(\widetilde{\Lambda}_2)S^{-1}(S^{2}(\Lambda_1))=\widetilde{\Lambda}_1\Lambda_2S(\widetilde{\Lambda}_2)S^{-1}(\Lambda_1)=\Lambda^{\{2\}}.
\end{align*}
Now $\chi_{V^*}(\Lambda^{\{2\}})=\chi_{V}(S(\Lambda^{\{2\}}))=\chi_{V}(\Lambda^{\{2\}})$, as desired. 

(2) Recall that the left and right adjoint representations of $H$ are the $\mathbbm{k}$-linear maps $\operatorname{ad}_l,\ \operatorname{ad}_r:H\rightarrow \operatorname{End}_{\mathbbm{k}}(H)$ given respectively by
$$\operatorname{ad}_l(a):H\rightarrow H,\ \operatorname{ad}_l(a)(b)=a_1bS(a_2),$$ $$\operatorname{ad}_r(a):H\rightarrow H,\ \operatorname{ad}_r(a)(b)=S(a_1)ba_2.$$ 
Note that $S\circ \operatorname{ad}_r(a)\circ S^{-1}=\operatorname{ad}_l(S(a))$ for any $a\in H$. Now
\begin{align*}
\lambda(\Lambda^{\{2\}})&=\lambda(\widetilde{\Lambda}_1\Lambda_2S(\widetilde{\Lambda}_2)S^{-1}(\Lambda_1))
=\lambda(S(\widetilde{\Lambda}_2)S^{-1}(\Lambda_1)S^{-2}(\widetilde{\Lambda}_1)S^{-2}(\Lambda_2))\ \text{by}\ (\ref{equ4004})\\
&=\lambda(S(\widetilde{\Lambda}_2)S(\Lambda_1)S^{-2}(\widetilde{\Lambda}_1)\Lambda_2)
=\text{Tr}(\operatorname{ad}_r(\Lambda)\circ S^{-2})\ \text{by}\ (\ref{equ400444})\\
&=\text{Tr}(S^{-1}\circ(S\circ \operatorname{ad}_r(\Lambda)\circ S^{-1})\circ S^{-1})=\text{Tr}(S^{-1}\circ \operatorname{ad}_l(S(\Lambda))\circ S^{-1})\\
&=\text{Tr}(\operatorname{ad}_l(\Lambda)\circ S^{-2})=\lambda(S(\widetilde{\Lambda}_2)\Lambda_1S^{-2}(\widetilde{\Lambda}_1)S(\Lambda_2))\ \text{by}\ (\ref{equ400444})\\
&=\lambda(\Lambda_1S^{-2}(\widetilde{\Lambda}_1)S(\Lambda_2)S^{-1}(\widetilde{\Lambda}_2))\ \text{by}\ (\ref{equ4004})\\
&=\lambda(\Lambda_1\widetilde{\Lambda}_1S(\Lambda_2)S(\widetilde{\Lambda}_2))=\lambda(S(\widetilde{\Lambda}_2)\Lambda_1\widetilde{\Lambda}_1S(\Lambda_2))\ \text{by}\ (\ref{equ004})\\
&=\text{Tr}(\operatorname{ad}_l(\Lambda))\ \text{by}\ (\ref{equ400444})\\
&=\chi_{\operatorname{ad}}(\Lambda).
\end{align*}
This completes the proof.
\qed

\begin{rem}
If $H$ is not semisimple, then $\Lambda^2=0$ implies that $\operatorname{ad}_l(\Lambda):H\rightarrow H$ is nilpotent. In this case, $\lambda(\Lambda^{\{2\}})=\chi_{\operatorname{ad}}(\Lambda)=0$. If $H$ is semisimple and we choose $\Lambda$ to be idempotent, then $\lambda(\Lambda)=1$ implies that $\lambda=\chi_{\operatorname{reg}}$. In this case,  $\chi_{\operatorname{reg}}(\Lambda^{\{2\}})=\chi_{\operatorname{ad}}(\Lambda^{\{1\}}).$ We shall prove later that $\chi_{\operatorname{reg}}(\Lambda^{\{m\}})=\chi_{\operatorname{ad}}(\Lambda^{\{m-1\}})$ for all $m>1$.
\end{rem}

For any finite-dimensional Hopf algebra $H$, the characters of the (left) regular
representation and the (left) adjoint representation of $H$ are both closely related to a right integral of $H^*$. These relationships are investigated in the following proposition:  
\begin{prop}\label{t5}
Let $H$ be a finite-dimensional (not necessary unimodular) Hopf algebra over a field $\mathbbm{k}$ with a left integral $\Lambda\in H$ and a right integral $\lambda\in H^*$ such that $\lambda(\Lambda)=1$. Let $g$ be the distinguished group-like
element of $H$ and $\alpha$ the distinguished group-like
element of $H^*$. We have
\begin{enumerate}
\item $\Lambda_1\alpha^{-1}(S(\Lambda_2))g^{-1}S^3(\Lambda_3)g\rightharpoonup\lambda=\chi_{\operatorname{reg}}=\lambda\leftharpoonup\alpha(\Lambda_1)S^2(\Lambda_2)S(\Lambda_3)$.
\item  $\lambda(S(\widetilde{\Lambda}_2)S(\Lambda_3)\widetilde{\Lambda}_1S^2(\Lambda_2))\Lambda_1\rightharpoonup \lambda=\chi_{\operatorname{ad}}=\lambda\leftharpoonup \lambda(S(\widetilde{\Lambda}_2)\Lambda_1\widetilde{\Lambda}_1S(\Lambda_2))S(\Lambda_3)$, where $\widetilde{\Lambda}=\Lambda$.
\end{enumerate}
\end{prop}
\proof
(1) Suppose $\chi_{\operatorname{reg}}=\lambda\leftharpoonup h$ for some $h\in H$. For any $a\in H$, we have
\begin{equation}\label{t4}\lambda(ha)=\chi_{\operatorname{reg}}(a)=\text{Tr}(L_a)=\lambda(S(\Lambda_2)a\Lambda_1)=\lambda(\alpha(\Lambda_1)S^2(\Lambda_2)S(\Lambda_3)a),\end{equation}
where the latter follows from \cite[Theorem 3(a)]{Rad}. By the non-degeneracy of $\lambda$, we have $h=\alpha(\Lambda_1)S^2(\Lambda_2)S(\Lambda_3).$
Similarly, using \cite[Theorem 3(b)]{Rad} one can prove $\Lambda_1\alpha^{-1}(S(\Lambda_2))g^{-1}S^3(\Lambda_3)g\rightharpoonup\lambda=\chi_{\operatorname{reg}}$.

(2) Suppose $\chi_{\operatorname{ad}}=\lambda\leftharpoonup h$ for some $h\in H$. For any $a\in H$, we have
\begin{equation}\label{t3}\lambda(ha)=\chi_{\operatorname{ad}}(a)=\text{Tr}(\operatorname{ad}_l(a))=\lambda(S(\widetilde{\Lambda}_2)a_1\widetilde{\Lambda}_1S(a_2)).\end{equation}
%where the latter follows from (\ref{equ400444}). 
Note that  $h=\lambda(h\Lambda_1)S(\Lambda_2)$ by (\ref{equ4004444}). Replacing $a$ with $\Lambda_1$ in (\ref{t3}), we have
$$h=\lambda(h\Lambda_1)S(\Lambda_2)=\lambda(S(\widetilde{\Lambda}_2)\Lambda_1\widetilde{\Lambda}_1S(\Lambda_2))S(\Lambda_3).$$
The proof of the equality $\lambda(S(\widetilde{\Lambda}_2)S(\Lambda_3)\widetilde{\Lambda}_1S^2(\Lambda_2))\Lambda_1\rightharpoonup \lambda=\chi_{\operatorname{ad}}$ is similar.
\qed

\begin{rem}\label{rem1}
If $H$ is unimodular, then $\alpha=\varepsilon$ and $S(\Lambda)=\Lambda.$ The result of Proposition \ref{t5}(1) can be reduced to be  $S(\Lambda_2)\Lambda_1\rightharpoonup\lambda=\chi_{\operatorname{reg}}=\lambda\leftharpoonup S(\Lambda_2)\Lambda_1$. Later we will use this formula to compute the values $\chi_{\operatorname{reg}}(\Lambda^{\{2\}})$ for some restricted quantum groups.
\end{rem}

\section{\bf Semisimple case}

Let $H$ be a semisimple Hopf algebra over an algebraically closed field $\mathbbm{k}$ of characteristic 0 with an idempotent integral $\Lambda$. Let $\{V_0,\cdots,V_{n-1}\}$ be a complete set of non-isomorphic simple $H$-modules. Let $\{e_0,\cdots,e_{n-1}\}$ and $\{\chi_0,\cdots,\chi_{n-1}\}$ be the associated central primitive idempotents and irreducible characters of $H$ respectively. Denote $d_i=\dim_{\mathbbm{k}}(V_i)$ for $i=0,\cdots,n-1$. The antipode of $H^*$ is denoted by $s$. In this case,
we will give the expression of $\Lambda^{\{m\}}$ which is slightly different from that of \cite[Proposition 2.2]{CW}. Then we use the expression of $\Lambda^{\{m\}}$ to describe the gauge invariant $\chi_V(\Lambda^{\{m\}})$ for any simple $H$-module $V$ and any $m>0$. In particular, we point out that $\chi_V(\Lambda^{\{m\}})$ can be expressed via the fusion rules of the fusion category $H$-mod.

\begin{thm}\label{t2}
Let $H$ be a semisimple Hopf algebra over an algebraically closed field $\mathbbm{k}$ of characteristic 0 with an idempotent integral $\Lambda$ of $H$. 
\begin{enumerate}
\item For $a\in Z(H)$, we have $\{a,\Lambda\}=\sum_{i=0}^{n-1}\frac{1}{d_i^2}\langle\chi_is(\chi_i),a\rangle e_i.$
\item For $m>1$, we have $\Lambda^{\{m\}}=\sum_{i=0}^{n-1}\frac{1}{d_i^2}\langle\chi_is(\chi_i),\Lambda^{\{m-1\}}\rangle e_i.$
\item For $m>0$, the commutators $\Lambda^{\{m\}}$ can be described as follows:

$$\Lambda^{\{m\}}=\left\{
    \begin{array}{ll}
      e_0, & \hbox{$m=1$;} \\
      \sum_{i=0}^{n-1}\frac{1}{d_i^2}e_i, & \hbox{$m=2$;} \\
      \sum_{i_1,\cdots,i_{m-1}=0}^{n-1}\frac{1}{d^2_{i_1}\cdots d^2_{i_{m-1}}}\\ \times\langle\chi_{i_2}s(\chi_{i_2}),e_{i_1}\rangle\cdots \langle\chi_{i_{m-1}}s(\chi_{i_{m-1}}),e_{i_{m-2}}\rangle e_{i_{m-1}}, & \hbox{$m>2$.}
    \end{array}
  \right.
$$
\end{enumerate}
\end{thm}
\proof
(1) Note that $\chi_i\leftharpoonup z=\frac{1}{d_i}\langle\chi_i,z\rangle\chi_i$ for $z\in Z(H)$ (see \cite[Eq.(13)]{CW}). For $a\in Z(H)$, we have $\{a,\Lambda\}\in Z(H)$. Then
\begin{align*}
\langle\chi_i,\{a,\Lambda\}\rangle&=\langle\chi_i,a_1\Lambda_1S(a_2)S(\Lambda_2)\rangle\\
&=\langle\chi_i\leftharpoonup \Lambda_1S(a_2)S(\Lambda_2),a_1\rangle\ \ \ \ \ (\text{since}\  \Lambda_1S(a_2)S(\Lambda_2)\in Z(H))\\
&=\frac{1}{d_i}\langle\chi_i,\Lambda_1S(a_2)S(\Lambda_2)\rangle\langle\chi_i,a_1\rangle
=\frac{1}{d_i}\langle\chi_i,S(\Lambda_2)\Lambda_1S(a_2)\rangle\langle\chi_i,a_1\rangle\\
&=\frac{1}{d_i}\langle\chi_i,S(a_2)\rangle\langle\chi_i,a_1\rangle
=\frac{1}{d_i}\langle\chi_is(\chi_i),a\rangle.
\end{align*}
It follows that $\{a,\Lambda\}=\sum_{i=0}^{n-1}\frac{1}{d^2_i}\langle\chi_is(\chi_i),a\rangle e_i.$

(2) Since $\Lambda^{\{m-1\}}\in Z(H)$ for $m>1$, by Part (1) we have $$\Lambda^{\{m\}}=\{\Lambda^{\{m-1\}},\Lambda\}=\sum_{i=0}^{n-1}\frac{1}{d_i^2}\langle\chi_is(\chi_i),\Lambda^{\{m-1\}}\rangle e_i,$$  as desired. 

(3) Obviously, $\Lambda^{\{1\}}=\Lambda=e_0$. By Part (2), we have  $$\Lambda^{\{2\}}=\sum_{i=0}^{n-1}\frac{1}{d_i^2}\langle\chi_is(\chi_i),\Lambda\rangle e_i=\sum_{i=0}^{n-1}\frac{1}{d_i^2}e_i,$$ see also \cite[Theorem 2.8]{Cohen} for the same formula. We proceed by induction on $m$ for $m>2$. For the case $m=3$, since $\Lambda^{\{2\}}=\sum_{i_1=0}^{n-1}\frac{1}{d_{i_1}^2}e_{i_1}$, it follows that
$$
\Lambda^{\{3\}}=\sum_{i_2=0}^{n-1}\frac{1}{d_{i_2}^2}\langle\chi_{i_2}s(\chi_{i_2}),\Lambda^{\{2\}}\rangle e_{i_2}=\sum_{i_1,i_2=0}^{n-1}\frac{1}{d_{i_1}^2d_{i_2}^2}\langle\chi_{i_2}s(\chi_{i_2}),e_{i_1}\rangle e_{i_2}.
$$ Assume that for the case $m=t$, we have $$\Lambda^{\{t\}}=\sum_{i_1,\cdots,i_{t-1}=0}^{n-1}\frac{1}{d^2_{i_1}\cdots d^2_{i_{t-1}}}\langle\chi_{i_2}s(\chi_{i_2}),e_{i_1}\rangle\cdots \langle\chi_{i_{t-1}}s(\chi_{i_{t-1}}),e_{i_{t-2}}\rangle e_{i_{t-1}}.$$ For the case $m=t+1$, we have
\begin{align*}
\Lambda^{\{t+1\}}&=\sum_{i_t=0}^{n-1}\frac{1}{d_{i_t}^2}\langle\chi_{i_t}s(\chi_{i_t}),\Lambda^{\{t\}}\rangle e_{i_{t}}\\
&=\sum_{i_1,\cdots,i_{t}=0}^{n-1}\frac{1}{d^2_{i_1}\cdots d^2_{i_{t}}}\langle\chi_{i_2}s(\chi_{i_2}),e_{i_1}\rangle\cdots \langle\chi_{i_{t}}s(\chi_{i_{t}}),e_{i_{t-1}}\rangle e_{i_{t}}
\end{align*}
This completes the proof.
\qed

%We proceed by induction on $m$. If $m=2$, then the equality $\langle\chi_is(\chi_i),\Lambda_{\{1\}}\rangle=\langle\chi_is(\chi_i),e_0\rangle=1$ shows that $\sum_{i=0}^{n-1}\frac{1}{d_i^2}\langle\chi_is(\chi_i),\Lambda_{\{1\}}\rangle e_i=\sum_{i=0}^{n-1}\frac{1}{d_i^2}e_i=\Lambda_{\{2\}}$. Assume that for the case $m=t$, we have $\Lambda_{\{t\}}=\sum_{i=0}^{n-1}\frac{1}{d_i^2}\langle\chi_is(\chi_i),\Lambda_{\{t-1\}}\rangle e_i.$ For the case $m=t+1$, using $\{e_i,\Lambda\}=\sum_{j=0}^{n-1}\frac{1}{d^2_j}\langle\chi_js(\chi_j),e_i\rangle e_j$ we have
%\begin{align*}
%\Lambda_{\{t+1\}}&=\{\Lambda_{\{t\}},\Lambda\}=\sum_{i=0}^{n-1}\frac{1}{d_i^2}\langle\chi_is(\chi_i),\Lambda_{\{t-1\}}\rangle\{e_i,\Lambda\}\\
%&=\sum_{i=0}^{n-1}\frac{1}{d_i^2}\langle\chi_is(\chi_i),\Lambda_{\{t-1\}}\rangle \sum_{j=0}^{n-1}\frac{1}{d^2_j}\langle\chi_js(\chi_j),e_i\rangle e_j\\
%&=\sum_{j=0}^{n-1}\frac{1}{d^2_j}\langle\chi_js(\chi_j),\Lambda_{\{t\}}\rangle e_j.
%\end{align*}
%Part (1) is proved by induction.

\begin{cor}\label{c1}
Let $H$ be a semisimple Hopf algebra over an algebraically closed field $\mathbbm{k}$ of characteristic 0 with an idempotent integral $\Lambda$ of $H$. 
\begin{enumerate}
\item For $m>1$, we have $\chi_i(\Lambda^{\{m\}})=\frac{1}{d_i}\langle\chi_is(\chi_i),\Lambda^{\{m-1\}}\rangle.$ Moreover, $\chi_{\operatorname{reg}}(\Lambda^{\{m\}})=\chi_{\operatorname{ad}}(\Lambda^{\{m-1\}})$.
\item For $m>0$, $\chi_i(\Lambda^{\{m\}})$ can be described as follows:
$$\chi_i(\Lambda^{\{m\}})=\left\{
    \begin{array}{ll}
      \delta_{i0}, & \hbox{$m=1$;} \\
      \frac{1}{d_i}, & \hbox{$m=2$;} \\
       \frac{1}{d_i}\sum_{i_1,\cdots,i_{m-2}=0}^{n-1}\frac{1}{d^2_{i_1}\cdots d^2_{i_{m-2}}}\\ \times\langle\chi_{i_2}s(\chi_{i_2}),e_{i_1}\rangle\cdots \langle\chi_{i_{m-2}}s(\chi_{i_{m-2}}),e_{i_{m-3}}\rangle\langle\chi_{i}s(\chi_{i}),e_{i_{m-2}}\rangle , & \hbox{$m>2$.}
    \end{array}
  \right.
$$
\end{enumerate}
\end{cor}
\proof
(1) The equality  $\chi_i(\Lambda^{\{m\}})=\frac{1}{d_i}\langle\chi_is(\chi_i),\Lambda^{\{m-1\}}\rangle$ follows from Theorem \ref{t2}(2). Using this equality, we have
\begin{align*}
\chi_{\operatorname{reg}}(\Lambda^{\{m\}})&=\sum_{i=0}^{n-1}d_i\chi_i(\Lambda^{\{m\}})=\sum_{i=0}^{n-1}\langle\chi_is(\chi_i),\Lambda^{\{m-1\}}\rangle=\chi_{\operatorname{ad}}(\Lambda^{\{m-1\}}).
\end{align*}

Part (2) follows from Theorem \ref{t2}(3).
\qed

\begin{rem}
It can be found in the proof of \cite[Proposition 2.1(1)]{CW} that $$\langle\chi_{i}s(\chi_{i}),e_{j}\rangle=d_j\langle\chi_{i}s(\chi_{i})s(\chi_{j}),\Lambda\rangle$$ for $0\leq i,j\leq n-1$. Using this equality we can rewrite  $\Lambda^{\{m\}}$ and $\chi_i(\Lambda^{\{m\}})$ for $m>2$ appeared in Theorem \ref{t2}(3) and Corollary \ref{c1}(3) respectively as 
$$\Lambda^{\{m\}}=\sum_{i_1,\cdots,i_{m-1}=0}^{n-1}\frac{1}{d_{i_1}\cdots d_{i_{m-2}} d^2_{i_{m-1}}}\langle\chi_{i_2}s(\chi_{i_2})s(\chi_{i_1}),\Lambda\rangle\cdots \langle\chi_{i_{m-1}}s(\chi_{i_{m-1}})s(\chi_{i_{m-2}}),\Lambda\rangle e_{i_{m-1}}
$$ and
$$\chi_i(\Lambda^{\{m\}})=\sum_{i_1,\cdots,i_{m-2}=0}^{n-1}\frac{1}{d_{i_1}\cdots d_{i_{m-1}}}\langle\chi_{i_2}s(\chi_{i_2})s(\chi_{i_1}),\Lambda\rangle\cdots \langle\chi_{i_{m-1}}s(\chi_{i_{m-1}})s(\chi_{i_{m-2}}),\Lambda\rangle.$$
The form of $\Lambda^{\{m\}}$ appearing here is a version of  \cite[Proposition 2.2]{CW}, where $\Lambda^{\{m\}}$ is determined in terms of the commutator matrix associated to the map $T:Z(H)\rightarrow Z(H), T(a)=\{a,\Lambda\}$. The expressions of $\chi_i(\Lambda^{\{m\}})$ show that 
the gauge invariants  $\chi_i(\Lambda^{\{m\}})$ for  $m>2$ are completely determined by the fusion rules of the semisimple monoidal category $H$-mod. More explicitly, suppose $\chi_i\chi_j=\sum_{k=0}^{n-1}N_{ij}^k\chi_k$, where $N_{ij}^k$ are structure coefficients of the fusion ring of  $H$-mod with respect to the basis $\{\chi_0,\chi_1,\cdots,\chi_{n-1}\}$. Then $$\chi_i(\Lambda^{\{m\}})=\frac{1}{d_i}\sum_{i_1,\cdots,i_{m-2}=0}^{n-1}\frac{1}{d_{i_1}d_{i_2}\cdots d_{i_{m-2}}} N_{i_2i_2^*}^{i_1}N_{i_3i_3^*}^{i_2}\cdots N_{i_{m-2}i_{m-2}^*}^{i_{m-3}}N_{ii^*}^{i_{m-2}}.$$
\end{rem}

\section{\bf Example: finite groups}
 In this section, we will present the gauge invariants $\lambda(\Lambda^{\{2\}})$ and  $\lambda(\Lambda^{\{3\}})$ for a finite group. Let $G$ be a finite group and $\Lambda=\sum_{g\in G}g$. The field $\mathbbm{k}$  is assumed to be the field of complex numbers. The right integral $\lambda$ of $(\mathbbm{k}G)^*$ with $\lambda(\Lambda)=1$ is defined by $\lambda(g)=\delta_{g,e}$, where $e$ is the identity of $G$. Let $\operatorname{Irr}(G)$ denote the set of irreducible characters of $G$.

Obviously,
$$\Lambda^{\{2\}}=\sum_{a,b\in G}[a,b],$$where $[a,b]=aba^{-1}b^{-1}.$ Note that
\[
\lambda([a,b]) =
\begin{cases}
1, & \text{if}\ [a,b]=e,\\
0, & \text{otherwise}.
\end{cases}
\]
The condition $[a,b]=e$ is equivalent to $ab=ba$. Therefore, \begin{equation}\label{s1}\lambda(\Lambda^{\{2\}})=\#\{(a,b)\in G\times G\mid ab=ba\},\end{equation} which counts the number of ordered commuting pairs $(a,b)\in G\times G$.
For a fixed element $a\in G$, all the elements $b\in G$ that commute with $a$ are precisely the elements of the centralizer $C_G(a)$. Thus, the number of commuting ordered pairs is
\[
\lambda(\Lambda^{\{2\}})=\sum_{a\in G} |C_G(a)|.
\]
Now we rewrite this sum using the conjugacy classes of $G$. Let
\[
\mathcal{C}_1, \mathcal{C}_2, \dots, \mathcal{C}_{k(G)}
\]
be the distinct conjugacy classes of $G$. Choose a representative $a_i\in \mathcal{C}_i$ for each $i=1,\dots,k(G)$. By the orbit-stabilizer theorem, we have
\[
|\mathcal{C}_i| = \frac{|G|}{|C_G(a_i)|}.
\]
For any $a\in \mathcal{C}_i$, we have $|C_G(a)|=|C_G(a_i)|$. Therefore,
\begin{align}\label{s2}
\nonumber\lambda(\Lambda^{\{2\}})&=\sum_{a\in G} |C_G(a)|
= \sum_{i=1}^{k(G)} \sum_{a\in \mathcal{C}_i} |C_G(a)|
= \sum_{i=1}^{k(G)} |\mathcal{C}_i| \cdot |C_G(a_i)|\\
&= \sum_{i=1}^{k(G)} \frac{|G|}{|C_G(a_i)|} \cdot |C_G(a_i)|
= \sum_{i=1}^{k(G)} |G|
= |G|\, k(G).
\end{align}

Observe that
\[
\Lambda^{\{3\}}=\sum_{a,b,c\in G}[[a,b],c],
\] where $[[a,b],c]=aba^{-1}b^{-1}cbab^{-1}a^{-1}c^{-1}$.
By the definition of $\lambda$, we have
\[
\lambda(\Lambda^{\{3\}})
= \#\{(a,b,c)\in G\times G\times G \mid [[a,b],c]=e\}.
\]
For fixed $a,b\in G$, the number of $c\in G$ satisfying $[[a,b],c]=e$ is exactly the size of the centralizer of the commutator $[a,b]$. Hence
\[
\lambda(\Lambda^{\{3\}})=\sum_{a,b\in G} |C_G([a,b])|.
\]

Now let $N(g)=\#\{(a,b)\in G\times G \mid [a,b]=g\}$. The Frobenius formula tells us that
\[
N(g)=|G|\sum_{\chi\in\operatorname{Irr}(G)} \frac{\chi(g)}{\chi(1)}.
\] 
Note that $
|C_G([a,b])| = \sum_{g\in G}\delta_{g,[a,b]} \, |C_G(g)|.
$ Substituting this into the expression for $\lambda(\Lambda^{\{3\}})$, we obtain 
\[
\lambda(\Lambda^{\{3\}})
= \sum_{a,b\in G} \sum_{g\in G}\delta_{g,[a,b]} \, |C_G(g)|= \sum_{g\in G}|C_G(g)|\left(\sum_{a,b\in G}\delta_{g,[a,b]} \right)
\]
The sum $\sum_{a,b\in G}\delta_{g,[a,b]}$ counts the number of pairs $(a,b)$ with $[a,b]=g$, i.e., it equals $N(g)$. Therefore,
\[
\lambda(\Lambda^{\{3\}})
= \sum_{g\in G} N(g) |C_G(g)|=|G| \sum_{\chi\in\operatorname{Irr}(G)} \frac{1}{\chi(1)}
  \sum_{g\in G} \chi(g) |C_G(g)|.
\]
Recall the following identity in finite group character theory:
\[
\begin{aligned}
\sum_{g\in G} \chi(g)|C_G(g)|
&= \sum_{g\in G} \chi(g) \sum_{\psi\in\operatorname{Irr}(G)} \psi(g)\overline{\psi(g)} \quad \text{(column orthogonality)}\\
&= \sum_{\psi\in\operatorname{Irr}(G)} \sum_{g\in G} \chi(g)\psi(g)\overline{\psi(g)}= \sum_{\psi\in\operatorname{Irr}(G)} |G| \langle \psi\overline{\psi}, \overline{\chi} \rangle\\
&= |G| \sum_{\psi\in\operatorname{Irr}(G)} \langle \psi\overline{\psi}, \overline{\chi} \rangle,
\end{aligned}
\] where $\langle \cdot,\cdot \rangle$ is the standard inner product on class functions.
Therefore,
\[
\lambda(\Lambda^{\{3\}})
= |G| \sum_{\chi\in\operatorname{Irr}(G)} \frac{1}{\chi(1)}
  \left( |G| \sum_{\psi\in\operatorname{Irr}(G)} \langle \psi\overline{\psi}, \overline{\chi} \rangle \right)
= |G|^2 \sum_{\chi,\psi\in\operatorname{Irr}(G)}
  \frac{\langle \psi\overline{\psi},\overline{ \chi }\rangle}{\chi(1)}.
\]

\section{\bf Example: quantum double of a finite group}

In this section, we will compute the gauge invariant $\lambda(\Lambda^{\{2\}})$ for a quantum double of a finite group. Let $G$ be a finite group with unit $e$, and let $\mathbbm{k}G$ be the group algebra over a field $\mathbbm{k}$ of characteristic zero. The dual space $(\mathbbm{k}G)^*$ has a basis given by the delta functions $\delta_g$ for $g\in G$, i.e., $\delta_g(g')=\delta_{g,g'}$ for any $g'\in G$. With the multiplication $\delta_g\delta_h=\delta_{g,h}\delta_h$, the dual $(\mathbbm{k}G)^*$ is an algebra with unit $\sum_{g\in G}\delta_g$. Let $D(G)$ be the Drinfeld double of $G$. It is a semisimple Hopf algebra with a basis given by $\{\delta_h\bowtie g\mid h, g\in G\}.$
The multiplication rule in $D(G)$ is
\[
(\delta_h\bowtie  g)(\delta_a\bowtie  b) = \delta_h \delta_{g a g^{-1}}\bowtie  g b = \delta_{g^{-1} h g, a} \delta_h\bowtie  g b\ \text{for}\ h,g,a,b\in G.
\]

The identity element of $D(G)$ is
\[
\sum_{g \in G} \delta_g\bowtie e.
\]
The comultiplication of $D(G)$ is given by
\[
\Delta(\delta_h\bowtie  g) = \sum_{a, b \in G, \, ab = h} (\delta_b\bowtie  g) \otimes (\delta_a\bowtie  g)\ \text{for}\ h,g\in G.
\]
The antipode of $D(G)$ is given by
\[
S(\delta_h\bowtie  g)= \delta_{g^{-1}h^{-1}g}\bowtie g^{-1}\ \text{for}\ h,g\in G.
\]
The counit of  $D(G)$ is given by
\[
\varepsilon(\delta_h \bowtie  g) = \delta_{h,e}\ \text{for}\ h,g\in G.
\]
As a Hopf algebra, $D(G)$ is unimodular with a two-sided integral $\Lambda=\sum_{g\in G}\delta_e\bowtie g$. It is easy to see that
$$\Delta(\Lambda)=\Lambda_1\otimes\Lambda_2=\sum_{g,h\in G}(\delta_h\bowtie g)\otimes(\delta_{h^{-1}}\bowtie g),$$
$$S(\Lambda_1)\otimes\Lambda_2=\sum_{g,h\in G}(\delta_{g^{-1}h^{-1}g}\bowtie g^{-1})\otimes(\delta_{h^{-1}}\bowtie g),$$
$$\Lambda_1\otimes S(\Lambda_2)=\sum_{g,h\in G}(\delta_{h}\bowtie g)\otimes(\delta_{g^{-1}hg}\bowtie g^{-1}).$$
Let $\widetilde{\Lambda}$ be a copy of $\Lambda.$
It is straightforward that 
\begin{align*}\Lambda^{\{2\}}&=\Lambda_1\widetilde{\Lambda}_1S(\Lambda_2)S(\widetilde{\Lambda}_2)\\
&=\sum_{a,b,g,h\in G}\delta_{g^{-1}hg,a}\delta_{b^{-1}g^{-1}hgb,g^{-1}hg}\delta_{gb^{-1}g^{-1}hgbg^{-1},b^{-1}ab}(\delta_h\bowtie gbg^{-1}b^{-1}).
\end{align*}
Note that a right integral $\lambda$ of $D(G)^*$ with $\lambda(\Lambda)=1$ has the form  $\lambda=\sum_{v\in G}v\bowtie \delta_e$. Therefore,
\begin{align*}\lambda(\Lambda^{\{2\}})&=
\sum_{v,a,b,g,h\in G}\delta_{g^{-1}hg,a}\delta_{b^{-1}g^{-1}hgb,g^{-1}hg}\delta_{gb^{-1}g^{-1}hgbg^{-1},b^{-1}ab}\langle v\bowtie \delta_e,\delta_h\bowtie gbg^{-1}b^{-1}\rangle\\
&=\sum_{v,a,b,g,h\in G}\delta_{g^{-1}hg,a}\delta_{b^{-1}g^{-1}hgb,g^{-1}hg}\delta_{gb^{-1}g^{-1}hgbg^{-1},b^{-1}ab}\langle v,\delta_h\rangle\langle\delta_e, gbg^{-1}b^{-1}\rangle\\
&=\sum_{a,b,g,h\in G}\delta_{g^{-1}hg,a}\delta_{b^{-1}g^{-1}hgb,g^{-1}hg}\delta_{gb^{-1}g^{-1}hgbg^{-1},b^{-1}ab}\delta_{gbg^{-1}b^{-1},e}\\
&=\sum_{b,g,h\in G}\delta_{b^{-1}g^{-1}hgb,g^{-1}hg}\delta_{gb^{-1}g^{-1}hgbg^{-1},b^{-1}g^{-1}hgb}\delta_{gbg^{-1}b^{-1},e}\\
&=\sum_{b,g,h\in G}\delta_{hb,bh}\delta_{gh,hg}\delta_{gb,bg},
\end{align*} 
which is the number of triples of pairwise commuting elements of $G$.  
Let $C_G(g)$ be the centralizer of $g$  and let $k(C_G(g))$ denote the number of conjugacy classes of the finite group $C_G(g)$. We have
\begin{align*}\lambda(\Lambda^{\{2\}})&=\#\{(g,h,b)\in G\times G\times G\mid gh=hg,hb=bh,gb=bg\}\\
&=\sum_{g\in G}\Bigl(\#\{(h,b)\in C_G(g)\times C_G(g)\mid hb=bh\} \Bigl)\\
&=\sum_{g\in G}|C_G(g)|k(C_G(g))\ \text{by}\ (\ref{s1})\ \text{and}\ (\ref{s2})
\end{align*}

We point out that if $G$ is an abelian group, then  $\lambda(\Lambda^{\{2\}})=|G|^3.$

\section{\bf Example: Restricted quantum groups}
In this section, we consider the restricted quantum group $\overline{U}_q(sl_2)$ at a primitive $2p$-th root of unity $q$.  We compute the twisting invariants $\chi_V(\Lambda^{\{2\}})$ and $\chi_{\text{reg}}(\Lambda^{\{2\}})$ respectively, where $V$ is a simple $\overline{U}_q(sl_2)$-module of dimension $p$ generated by a highest weight vector. It turns out that $\chi_{\text{reg}}(\Lambda^{\{2\}})/\chi_V(\Lambda^{\{2\}})=2p$, which is a trivial gauge invariant of $\overline{U}_q(sl_2)$.

\subsection{$q$-binomial coefficients}

Let $p$ be a positive integer and $q$ a primitive $2p$-th root of unity (so that $q^p=-1$).  We work with the standard $q$-analogues:
\[
[n]=\frac{q^n-q^{-n}}{q-q^{-1}},\qquad [n]!=[1][2]\cdots[n],\qquad [0]!=1.
\]
The $q$-binomial coefficient is defined by
\[
\begin{bmatrix} m \\ n \end{bmatrix}
=
\begin{cases}
0, & n < 0 \quad \text{or} \quad m - n < 0, \\[1.2ex]
\dfrac{[m]!}{[n]! \, [m - n]!}, & \text{otherwise.}
\end{cases}
\]

Since $q^p=-1$, we have $[p-k]=[k]$, hence
\begin{align*}
[p-1]! &= \prod_{i=1}^{p-1}[i] = \left(\prod_{i=1}^{s}[i]\right)\left(\prod_{i=s+1}^{p-1}[i]\right)
= [s]!\,\prod_{j=1}^{p-1-s}[p-j]\\
&=[s]!\,\prod_{j=1}^{p-1-s}[j]= [s]!\,[p-1-s]!.
\end{align*}
Thus, the $q$-binomial coefficient 
\begin{equation}\label{eq:binom}\begin{bmatrix} p-1 \\ s \end{bmatrix}=\frac{[p-1]!}{[s]![p-1-s]!}
=1\ \text{for}\ 0\le s\le p-1.\end{equation}

\subsection{Restricted quantum groups}

The restricted quantum group $\overline{U}_q(sl_2)$ is a Hopf algebra generated by $E,F,K$ over a field $\mathbbm{k}$ of characteristic zero subject to
\[
E^p=F^p=0,\qquad K^{2p}=1,\qquad KEK^{-1}=q^2E,\qquad KFK^{-1}=q^{-2}F,
\]
and
\[
[E,F]=\frac{K-K^{-1}}{q-q^{-1}}.
\]
The comultiplication, counit, and antipode are given respectively by
\[
\Delta(E)=1\otimes E+E\otimes K,\qquad
\Delta(F)=K^{-1}\otimes F+F\otimes 1,\qquad
\Delta(K)=K\otimes K,
\]
\[
\varepsilon(E)=\varepsilon(F)=0,\quad \varepsilon(K)=1,
\]
\[
S(E)=-EK^{-1},\qquad S(F)=-KF,\qquad S(K)=K^{-1}.
\]
It is well known that the algebra $\overline{U}_q(sl_2)$ has a basis $\{E^iK^jF^l\mid 0\le i,l\le p-1,\;0\le j\le 2p-1\}$, hence the dimension of $\overline{U}_q(sl_2)$ is $2p^3$, see e.g. \cite{FGST} for details.

For any integer $s$, we set
\[
[K;s]=\frac{Kq^s-K^{-1}q^{-s}}{q-q^{-1}}.
\]
The following basic properties are immediate from $q^p=-1$:
\begin{align}\label{e4}[K;s]E^{p-1}=E^{p-1}[K;s-2],\ \ \  [K^{-1};s]=-[K;-s]\end{align}
and
\begin{align}\label{e5}
[K;s+p]=\frac{Kq^{s+p}-K^{-1}q^{-(s+p)}}{q-q^{-1}}
=\frac{-Kq^{s}+K^{-1}q^{-s}}{q-q^{-1}}
=-[K;s].
\end{align}

The following commutation formulae are standard (see, e.g., \cite{KS}):

For $m\geq n\geq0$, we have
\begin{align}\label{e2}E^{n}F^{m}&=\sum_{i=0}^{n}\frac{[n]![m]!}{[i]![n-i]![m-i]!}\\
\nonumber &\quad \times F^{m-i}E^{n-i}[K;n-m][K;n-m-1]\cdots[K;n-m-i+1].
\end{align}
Applying the automorphism of $\overline{U}_q(sl_2)$ given by $E\mapsto F,F\mapsto E,K\mapsto K^{-1}$ to both sides of the equality (\ref{e2}), we have
\begin{align}\label{e3}F^{n}E^{m}&=\sum_{i=0}^{n}\frac{[n]![m]!}{[i]![n-i]![m-i]!}\\
\nonumber &\quad \times E^{m-i}F^{n-i}[K^{-1};n-m][K^{-1};n-m-1]\cdots[K^{-1};n-m-i+1].
\end{align}

From \cite[Eq.(3.1)]{FGST}, the  comultiplication of the monomial $F^m E^n K^j$ is
\[
\begin{aligned}
\Delta(F^m E^n K^j)
&= \sum_{r=0}^m \sum_{s=0}^n 
q^{2(n-s)(r-m)+r(m-r)+s(n-s)}
\begin{bmatrix} m \\ r \end{bmatrix}
\begin{bmatrix} n \\ s \end{bmatrix} \\
&\quad \times F^r E^{n-s} K^{r-m+j} \otimes F^{m-r} E^s K^{n-s+j}. 
\end{aligned}
\]

The left and right integrals of $\overline{U}_q(sl_2)$ coincide and have the form \cite[Eq.(3.2)]{FGST}: $$\Lambda=F^{p-1}E^{p-1}\sum_{j=0}^{2p-1}K^{j}.$$
Using (\ref{eq:binom}),  a straightforward computation gives
\begin{equation}\Delta(\Lambda)=\sum_{j=0}^{2p-1}\sum_{r,s=0}^{p-1}(-1)^{r+s}q^{-(r+s+1)(r+s+2)}F^{r}E^{p-1-s}K^{r+j+1-p}\otimes F^{p-1-r}E^{s}K^{p+j-1-s}.\end{equation}
Consequently, applying the antipode to the second tensor factor, we obtain
\begin{equation}\label{e02}\Lambda_1\otimes S(\Lambda_2)=\sum_{j=0}^{2p-1}\sum_{r,s=0}^{p-1}q^{-2(r+s+1)(r+j)}F^{r}E^{p-1-s}K^{r+j+1-p}\otimes E^sF^{p-1-r}K^{-r-j}\end{equation}
and similarly,
\begin{equation}\label{e002}S^{-1}(\Lambda_1)\otimes \Lambda_2=\sum_{j=0}^{2p-1}\sum_{r,s=0}^{p-1}q^{2(r+s+1)(j-s)}E^{p-1-s}F^rK^{s-j}\otimes F^{p-1-r}E^sK^{p-1-s+j}.\end{equation}
Note that $\Lambda^{\{2\}}=\{\Lambda,\widetilde{\Lambda}\}=\Lambda_1\widetilde{\Lambda}_2S(\Lambda_2)S^{-1}(\widetilde{\Lambda}_1)$, where $\widetilde{\Lambda}=\Lambda$.
Substituting (\ref{e02}) and (\ref{e002}) and simplifying yields
\begin{align}\label{e6}
\Lambda^{\{2\}}=4p^2\sum_{0\leq s\leq r \leq p-1}F^{p-1-s}E^{p-1-s}F^{p-1-r}E^{p-1-r+s}F^sE^rF^r.
\end{align}

\subsection{The twisting invariant $\chi_V(\Lambda^{\{2\}})$}
Let $V$ be a $p$-dimensional $\overline{U}_q(sl_2)$-module with a basis $\{v_0,\dots,v_{p-1}\}$. The $\overline{U}_q(sl_2)$-action is given by
\[
K v_n = q^{p-1-2n} v_n,\qquad
E v_n = [n][p-n] v_{n-1},\qquad
F v_n = v_{n+1},
\]
with $v_p=v_{-1}=0$.  This is a simple $\overline{U}_q(sl_2)$-module, see \cite[Sec.3.2.1]{FGST}.
We now compute $\chi_V(\Lambda^{\{2\}})$.

\begin{prop}We have
$
\chi_V(\Lambda^{\{2\}}) = 4p^3 ([p-1]!)^4.
$
\end{prop}

\begin{proof}
Substituting (\ref{e6}) and using $\chi(ab)=\chi(ba)$, we have
\begin{align*}
\chi_V(\Lambda^{\{2\}})&=4p^2\sum_{0\leq s\leq r \leq p-1}\chi_V(F^{p-1-s}E^{p-1-s}F^{p-1-r}E^{p-1-r+s}F^sE^rF^r)\\
&=4p^2\sum_{0\leq s\leq r \leq p-1}\chi_V(F^{p-1-s+r}E^{p-1-s}F^{p-1-r}E^{p-1-r+s}F^sE^r)\\
&=4p^2\sum_{s=0}^{p-1}\chi_V(F^{p-1}E^{p-1-s}F^{p-1-s}E^{p-1}F^sE^s)\\
&=4p^2[p-1]!\sum_{s=0}^{p-1}\chi_V(F^{p-1}[K;0][K;-1]\cdots[K;2+s-p]\\
&\quad \times E^{p-1}[K^{-1};0][K^{-1};-1]\cdots[K^{-1};1-s])\ \text{by}\ (\ref{e2})\ \text{and}\  (\ref{e3}) \\  
&=4p^2[p-1]!\sum_{s=0}^{p-1}\chi_V(F^{p-1}E^{p-1}[K;-2][K;-3]\cdots[K;s-p]\\
&\quad \times (-1)^s[K;0][K;1]\cdots[K;s-1])\ \text{by}\ (\ref{e4})\\
&=4p^2[p-1]!\sum_{s=0}^{p-1}\chi_V(F^{p-1}E^{p-1}T(K;s)),
\end{align*}
where $T(K;s)=[K;-2][K;-3]\cdots[K;s-p](-1)^s[K;0][K;1]\cdots[K;s-1]$. Note that $K v_{p-1}= -q v_{p-1}.$ The action of $T(K;s)$ on $v_{p-1}$ is
\begin{align*}
T(K;s)v_{p-1}=(-1)^s[-q;-2][-q;-3]\cdots[-q;s-p][-q;0][-q;1]\cdots[-q;s-1]v_{p-1}.
\end{align*}
Since $[-q;r]=[-r-1]$ and $[-r]=-[r]$, it follows that
\begin{align*}
T(K;s)v_{p-1}&=(-1)^s[1][2]\cdots[p-s-1][-1][-2]\cdots[-s]v_{p-1}\\
&=[1][2]\cdots[p-s-1][1][2]\cdots[s]v_{p-1}\\
&=[p-1]!v_{p-1}\ \text{by}\ (\ref{eq:binom}).
\end{align*}
Note that $E v_n = [n][p-n] v_{n-1}$ and
$F v_n = v_{n+1}$ for $0\leq n\leq p-1$. We obtain that
$$F^{p-1}E^{p-1}v_n=0\ \text{for}\  0\leq n\leq p-2\ \text{and}\ F^{p-1}E^{p-1}v_{p-1}=([p-1]!)^2v_{p-1}.$$ 
Thus
\[
\chi_V\Bigl( F^{p-1}E^{p-1}T(K;s)\Bigr)
=([p-1]!)^3.
\]
Therefore,
\[
\chi_V(\Lambda^{\{2\}})
=4p^2[p-1]!\sum_{s=0}^{p-1}([p-1]!)^3
=4p^3 ([p-1]!)^4.
\]
This completes the proof. 
\end{proof}

\subsection{The twisting invariant $\chi_{\operatorname{reg}}(\Lambda^{\{2\}})$}
Let $\lambda$ be the right integral of $\overline{U}_q(sl_2)^*$ such that $\lambda(\Lambda)=1$, where $\Lambda=F^{p-1}E^{p-1}\sum_{j=0}^{2p-1}K^{j}.$ By \cite[Sec.3.1.2]{FGST}, $\lambda$ is given by \begin{equation}\label{e9}
\lambda(F^iE^mK^n)=\delta_{i,p-1}\delta_{m,p-1}\delta_{n,p+1}.\end{equation}
According to Remark \ref{rem1}, we have
\[
\chi_{\operatorname{reg}}(\Lambda^{\{2\}})=\lambda(S(\Lambda_2)\Lambda_1\Lambda^{\{2\}}).
\]
We now use this formula to compute $\chi_{\operatorname{reg}}(\Lambda^{\{2\}})$.

\begin{prop}We have
$
\chi_{\operatorname{reg}}(\Lambda_{\{2\}}) = 8p^4([p-1]!)^4.
$
\end{prop}
\proof
First, from (\ref{e02}) we have
\begin{align*}
S(\Lambda_2)\Lambda_1&=\sum_{j=0}^{2p-1}\sum_{r,s=0}^{p-1}q^{-2(r+s+1)(r+j)}E^sF^{p-1-r}K^{-r-j}F^{r}E^{p-1-s}K^{r+j+1-p}\\
&=\sum_{j=0}^{2p-1}\sum_{r,s=0}^{p-1}E^sF^{p-1}E^{p-1-s}K^{1-p}\\
&=2p^2\sum_{s=0}^{p-1}E^sF^{p-1}E^{p-1-s}K^{1-p}
\end{align*}
Multiplying by $\Lambda^{\{2\}}$ (which is central) and using (\ref{e2}), we have
\begin{align*}
S(\Lambda_2)\Lambda_1\Lambda^{\{2\}}
&=2p^2\sum_{t=0}^{p-1}E^tF^{p-1}\Lambda^{\{2\}}E^{p-1-t}K^{1-p}\\
&=8p^4\sum_{t=0}^{p-1}\sum_{0\leq s\leq r \leq p-1}E^tF^{p-1}F^{p-1-s}E^{p-1-s}F^{p-1-r}E^{p-1-r+s}F^sE^rF^rE^{p-1-t}K^{1-p}\\
&=8p^4\sum_{t=0}^{p-1}E^tF^{p-1}(E^{p-1}F^{p-1})(E^{p-1}F^{p-1})E^{p-1-t}K^{1-p}\\
&=8p^4([p-1]!)^2\sum_{t=0}^{p-1}E^tF^{p-1}([K;0][K;-1]\cdots[K;2-p])^2E^{p-1-t}K^{1-p}.
\end{align*}
Now
\begin{align*}
\chi_{\operatorname{reg}}(\Lambda^{\{2\}})&=\lambda(S(\Lambda_2)\Lambda_1\Lambda^{\{2\}})\\
&=8p^4([p-1]!)^2\sum_{t=0}^{p-1}\lambda(E^tF^{p-1}([K;0][K;-1]\cdots[K;2-p])^2E^{p-1-t}K^{1-p})\\
&=8p^4([p-1]!)^2\sum_{t=0}^{p-1}\lambda(F^{p-1}([K;0][K;-1]\cdots[K;2-p])^2\\
&\qquad \times  E^{p-1-t}K^{1-p}S^{-2}(E^t))\ \text{by}\ (\ref{equ4004})\\
&=8p^4([p-1]!)^2\sum_{t=0}^{p-1}\lambda(F^{p-1}([K;0][K;-1]\cdots[K;2-p])^2E^{p-1}K^{1-p})\\
&=8p^5([p-1]!)^2\lambda(F^{p-1}E^{p-1}([K;-2][K;-3]\cdots[K;-p])^2K^{1-p})\ \text{by}\ (\ref{e4})\\
&=8p^5([p-1]!)^2\lambda(F^{p-1}E^{p-1}P(K)^2K^{1-p}),
\end{align*}
where $P(K)$ is the Laurent polynomial with the variable $K$:
\[
P(K)=\prod_{r=-p}^{-2}[K;r].
\]
Let $\mu$ be the coefficient of \(K^0\) in \(P(K)^2\). According to the definition of $\lambda$ given in (\ref{e9}), we see that $$\chi_{\operatorname{reg}}(\Lambda^{\{2\}})=8p^5([p-1]!)^2\mu.$$
So we need to compute the constant term $\mu$ of \(P(K)^2\).

We work in the Laurent polynomial algebra \(\mathbbm{k}[K,K^{-1}]\) with \(K^{2p}=1\). In this setting, any Laurent polynomial $f(K)$ in \(\mathbbm{k}[K,K^{-1}]\) has the form \(f(K)=\sum_{n=0}^{2p-1} a_n K^n\).  The constant term $a_0$ of \(f(K)=\sum_{n=0}^{2p-1} a_n K^n\) can be extracted by averaging over the \(2p\)-th roots of unity. This follows from the discrete Fourier transform:
\[
a_0=\frac{1}{2p}\sum_{m=0}^{2p-1} f(q^m),
\]
where \(q\) is a primitive \(2p\)-th root of unity. Applying this to $f(K)=P(K)^2$, we have
\begin{equation}\label{e8}
\mu=\frac{1}{2p}\sum_{m=0}^{2p-1} P(q^m)^2.
\end{equation}

Now, we evaluate \(P(q^m)\) for  \(0 \le m \le 2p-1\). Substituting \(K=q^m\), we get
\[
P(q^m)=\prod_{r=-p}^{-2}\frac{q^{m+r}-q^{-(m+r)}}{q-q^{-1}}
=\prod_{r=-p}^{-2} [m+r]=\prod_{n=m-p}^{m-2} [n].
\]
Recall that \([n]=0\) if and only if \(q^n=q^{-n}\), which is equivalent to \(q^{2n}=1\). Since \(q\) is a  primitive \(2p\)-th root of unity, this occurs if and only if \(2p \mid 2n\), i.e. \(p \mid n\).

Consider the set
\begin{equation}\label{e7}
\{m-p,\ m-p+1,\ \dots,\ m-2\}
\end{equation}
consisting of \(p-1\) consecutive integers. If we add the number \(m-1\), we obtain \(p\) consecutive integers $\{m-p,\ m-p+1,\ \dots,\ m-2,\ m-1\}$ in which there is only one integer divisible by \(p\). Hence, for the product $\prod_{n=m-p}^{m-2} [n]$ to be non-zero, none of the numbers in (\ref{e7}) can be divisible by \(p\). Equivalently, the number $m-1$ must be divisible by \(p\). Thus,   
\[
m-1 \equiv 0 \pmod p.
\]
Within the range \(0 \le m \le 2p-1\), this gives exactly two solutions:
\[
m=1 \quad\text{and}\quad m=p+1.
\]
Hence, \(P(q^m)\neq0\) if and only if $m=1$ or $m=p+1$.

For these two cases $m=1$ and $m=p+1$, we compute:
$$
P(q) = \prod_{n=1-p}^{-1}[n] = \prod_{n=1}^{p-1}[-n] = (-1)^{p-1}\prod_{n=1}^{p-1}[n]= (-1)^{p-1}[p-1]!,
$$
$$
P(q^{p+1}) = \prod_{n=1}^{p-1}[n]=[p-1]!.
$$
Then \(P(q)^2=P(q^{p+1})^2 = ([p-1]!)^2\). Substituting back into the averaging formula (\ref{e8}) yields
\[
\mu = \frac{1}{2p}\left(([p-1]!)^2 + ([p-1]!)^2\right) = \frac{([p-1]!)^2}{p}.
\]
Now
\begin{align*}
\chi_{\operatorname{reg}}(\Lambda^{\{2\}})=8p^5([p-1]!)^2\mu=8p^4([p-1]!)^4.   
\end{align*}
This completes the proof
\qed

\begin{rem}We have the ratio
\begin{align*}
\frac{\chi_{\operatorname{reg}}(\Lambda^{\{2\}})}{\chi_{V}(\Lambda^{\{2\}})}=\frac{8p^4([p-1]!)^4}{4p^3([p-1]!)^4}=2p,
\end{align*}
which is a gauge invariant of $\overline{U}_q(sl_2)$.
\end{rem}

\section*{Acknowledgement}
The second author was supported by the National Natural Science Foundation of China (Grant No. 12371041).

\end{document}